\documentclass[11pt, reqno]{amsart}

\usepackage[utf8]{inputenc}
\usepackage[T1]{fontenc}
\usepackage[english]{babel}
\usepackage{amsmath, amsthm}
\usepackage{ae}
\usepackage{icomma}
\usepackage{units}
\usepackage{color}
\usepackage{graphicx}
\usepackage{bbm}
\usepackage{caption}
\usepackage{array}
\usepackage[hmarginratio=1:1]{geometry}
\usepackage[hyphens]{url}
\usepackage[pdfpagelabels=false, hidelinks]{hyperref}
\usepackage{mathrsfs}
\usepackage{amssymb}
\usepackage{placeins} 
\usepackage{tikz}
\usetikzlibrary{patterns}
\usepackage{float} 
\usepackage[nodayofweek]{datetime}
\usepackage{enumitem}
\usepackage{mathtools}
\usepackage[numbers, sort]{natbib}
\usepackage{dsfont}
\usepackage{ifthen}

\newcommand{\R}{\ensuremath{\mathbb{R}}}
\newcommand{\N}{\ensuremath{\mathbb{N}}}

\newcommand{\Z}{\ensuremath{\mathbb{Z}}}

\DeclareMathOperator{\Tr}{Tr}

\let\phi\varphi

\newcommand{\1}{\ensuremath{\mathds{1}}}

\DeclareMathOperator{\supp}{supp}
\newcommand{\Haus}{\ensuremath{\mathcal{H}}}

\newtheorem{theorem}{Theorem}[section]

\newtheorem{lemma}[theorem]{Lemma}
\newtheorem{proposition}[theorem]{Proposition}
\newtheorem{corollary}[theorem]{Corollary}

\numberwithin{theorem}{section}
\numberwithin{definition}{section}
\theoremstyle{remark}
\newtheorem{remark}[theorem]{Remark}

\newcommand{\limplus}{{\mathchoice{\vcenter{\hbox{$\scriptstyle +$}}}
  {\vcenter{\hbox{$\scriptstyle +$}}}
  {\vcenter{\hbox{$\scriptscriptstyle +$}}}
  {\vcenter{\hbox{$\scriptscriptstyle +$}}}
}}
\newcommand{\limminus}{{\mathchoice{\vcenter{\hbox{$\scriptstyle -$}}}
  {\vcenter{\hbox{$\scriptstyle -$}}}
  {\vcenter{\hbox{$\scriptscriptstyle -$}}}
  {\vcenter{\hbox{$\scriptscriptstyle -$}}}
}}
\newcommand{\limpm}{{\mathchoice{\vcenter{\hbox{$\scriptstyle \pm$}}}
  {\vcenter{\hbox{$\scriptstyle \pm$}}}
  {\vcenter{\hbox{$\scriptscriptstyle \pm$}}}
  {\vcenter{\hbox{$\scriptscriptstyle \pm$}}}
}}

\begin{document}

\title[On the error in the two-term Weyl formula]{On the error in the two-term Weyl formula\\ for the Dirichlet Laplacian}

\author[R. L. Frank]{Rupert L. Frank}
\address{\textnormal{(R. L. Frank)} Mathematisches Institut, Ludwig-Maximilans Universit\"at M\"unchen, Theresinstr. 39, 80333 M\"unchen, Germany, and Department of Mathematics, California Institute of Technology, Pasadena, CA 91125, USA}
\email{r.frank@lmu.de, rlfrank@caltech.edu}

\author[S. Larson]{Simon Larson}
\address{\textnormal{(S. Larson)} Department of Mathematics, California Institute of Technology, Pasadena, CA 91125, USA}
\email{larson@caltech.edu}


\thanks{\copyright\, 2020 by the authors. This paper may be
reproduced, in its entirety, for non-commercial purposes.\\
U.S.~National Science Foundation grant DMS-1363432 (R.L.F.) and Knut and Alice Wallenberg Foundation grant KAW~2018.0281 (S.L.) is acknowledged. The authors also wish to thank Institut Mittag-Leffler, where part of this work was carried out.}

\begin{abstract} 
  We study the optimality of the remainder term in the two-term Weyl law for the Dirichlet Laplacian within the class of Lipschitz regular subsets of $\R^d$. In particular, for the short-time asymptotics of the trace of the heat kernel we prove that the error term cannot be made quantitatively better than little-$o$ of the second term.
\end{abstract}

\maketitle

\section{Introduction and main results}

Let $-\Delta_\Omega$ denote the Dirichlet Laplace operator on an open set $\Omega\subset \R^d$, which is defined as a self-adjoint operator in $L^2(\Omega)$ through the quadratic form $u\mapsto \int_\Omega |\nabla u(x)|^2\, dx $ with form domain $H^1_0(\Omega)$. 
If the measure of $\Omega\subset \R^d$ is finite the spectrum of $-\Delta_\Omega$ is discrete and consists of an infinite number of positive eigenvalues accumulating only at infinity. Here the eigenvalues are denoted by
\begin{equation*}
    0<\lambda_1 \leq \lambda_2 \leq \lambda_3 \leq \ldots, 
\end{equation*}  
where each eigenvalue is repeated according to its multiplicity.

The study of the asymptotic behaviour of $\lambda_k$ as $k\to \infty$ is a classical topic in spectral theory. The most fundamental result in this area is the following celebrated result going back to Weyl~\cite{WeylAsymptotic} which states that
\begin{equation}\label{eq:Weyls law}
  \#\{\lambda_k <\lambda\} = \frac{\omega_d}{(2\pi)^d}|\Omega|\lambda^{d/2}+o(\lambda^{d/2})\quad \mbox{as }\lambda \to \infty\, .
\end{equation}
Here and in what follows $\omega_d$ denotes the volume of the $d$-dimensional unit ball. That~\eqref{eq:Weyls law} holds for any open set $\Omega\subset \R^d$ of finite measure was obtained in~\cite{zbMATH03421305}.

If the set $\Omega$ has certain geometric properties a refined version of the asymptotic expansion~\eqref{eq:Weyls law} holds, namely, 
\begin{equation}\label{eq:Ivrii_asymptotics}
  \#\{\lambda_k<\lambda\} = \frac{\omega_d}{(2\pi)^d}|\Omega|\lambda^{d/2}-\frac{1}{4}\frac{\omega_{d-1}}{(2\pi)^{d-1}}\Haus^{d-1}(\partial\Omega)\lambda^{(d-1)/2}+o(\lambda^{(d-1)/2}) \quad \mbox{as }\lambda \to \infty\, . 
\end{equation}
Here and in what follows $\Haus^{d-1}(A)$ denotes the $(d-1)$-dimensional Hausdorff measure of a set $A\subset \R^d$.
This refinement of Weyl's law was conjectured already by Weyl in~1913~\cite{MR1580880}. A satisfactory answer remained elusive for several decades, but in 1980 Ivrii~\cite{MR575202} proved the conjecture under the assumption that $\Omega$ is smooth and the measure of the periodic billiards in $\Omega$ is zero.

\subsection{Main results} In this paper our focus is on the remainder term in~\eqref{eq:Ivrii_asymptotics}, or rather the corresponding remainder term in certain averages of the counting function. The greater part of our analysis concerns the remainder term in the Abel-type average
\begin{equation}\label{eq:Brown asymptotics}
  \Tr(e^{t\Delta_\Omega}) = \sum_{k\geq 1}e^{-t\lambda_k} = (4\pi t)^{-d/2}\biggl(
    |\Omega|- \frac{\sqrt{\pi t}}{2}\Haus^{d-1}(\partial\Omega)+o(\sqrt{t})
  \biggr) \quad \mbox{as }t\to 0^\limplus\, , 
\end{equation}
that is, the short-time asymptotics of the trace of the heat kernel. The asymptotics~\eqref{eq:Brown asymptotics} can be obtained from~\eqref{eq:Ivrii_asymptotics} by integration in $\lambda$. However, it is not possible to reverse this process and deduce~\eqref{eq:Ivrii_asymptotics} from~\eqref{eq:Brown asymptotics} alone. 

Intuitively, the averaging of the eigenvalues should have a regularizing effect on the asymptotics and thus one expects~\eqref{eq:Brown asymptotics} to be valid under less restrictive geometric assumptions than those needed for~\eqref{eq:Ivrii_asymptotics}. Under the weak assumption that the boundary of $\Omega$ is Lipschitz regular the validity of~\eqref{eq:Brown asymptotics} was proved by Brown~\cite{MR1134755}. In the same paper Brown remarks without proof that the error term $o(\sqrt{t})$ cannot be replaced by $o(t^{1/2+\epsilon}), $ for any $\epsilon>0$. The main theorem of this paper goes in the same direction as this remark, and in fact contains the remark as a particular case. However, our result claims substantially more. While Brown's remark concerns the impossibility of improving the error term on the algebraic scale, our result states that it is impossible to make any quantitative improvement whatsoever. Specifically, we prove the following:
\begin{theorem}\label{thm:MainHeat}
   Let $g\colon \R_\limplus \to \R$ be a non-negative function with $\lim_{t\to 0^\limplus}g(t)=0$.
   There exists an open, bounded, and connected set $\Omega\subset \R^d$ with Lipschitz regular boundary such that
\begin{equation}\label{eq:counter_example}
  \lim_{t\to 0^\limplus} \frac{(4\pi t)^{d/2}\Tr(e^{t\Delta_\Omega})-|\Omega|+ \frac{\sqrt{\pi t}}{2}\Haus^{d-1}(\partial \Omega)}{\sqrt{t}g(t)}=\infty\, .
\end{equation}
\end{theorem}

In addition to considering the trace of the heat kernel~\eqref{eq:Brown asymptotics} we will consider the so-called Riesz means of order $\gamma\geq 0$, which are defined by
\begin{equation*}
  \Tr(-\Delta_\Omega-\lambda)_\limminus^\gamma = \sum_{\lambda_k<\lambda}(\lambda-\lambda_k)^\gamma\, , \quad \mbox{for }\lambda \geq 0 \, , 
\end{equation*}
where $x_\limpm = \frac{1}{2}(|x|\pm x)$. In particular, $\Tr(-\Delta_\Omega-\lambda)^0_\limminus= \#\{\lambda_k <\lambda\}$. The quantities $\Tr(-\Delta_\Omega-\lambda)_\limminus^\gamma$ become more and more well-behaved as $\gamma$ increases. Furthermore, by setting $\gamma = t\lambda$ and renormalising appropriately one obtains $\Tr(e^{t\Delta_\Omega})$ in the limit $\lambda\to \infty$.

Again by integration in $\lambda$ one can deduce a two-term asymptotic formula for the Riesz means from~\eqref{eq:Ivrii_asymptotics} which reads
\begin{equation}\label{eq: Reisz asymptotics}
  \Tr(-\Delta_\Omega-\lambda)_\limminus^\gamma = L_{\gamma, d}|\Omega|\lambda^{\gamma+d/2}- \frac{L_{\gamma, d-1}}{4}\Haus^{d-1}(\partial\Omega)\lambda^{\gamma + (d-1)/2} + o(\lambda^{\gamma+(d-1)/2})
\end{equation}
as $\lambda \to \infty$, where we abbreviate 
$$L_{\gamma, d}= \frac{\Gamma(\gamma+1)}{(4\pi)^{d/2}\Gamma(\gamma+1 +d/2)}\, .$$
As in the case of the trace of the heat kernel, the regularizing effect of averaging should help to prove the validity of~\eqref{eq: Reisz asymptotics} in greater generality for larger $\gamma$. In a recent paper by the authors~\cite{FrankLarson} it is proved that~\eqref{eq: Reisz asymptotics} is valid for $\gamma\geq 1$ as soon as the boundary of $\Omega$ is Lipschitz regular. This implies Brown's result~\cite{MR1134755} but not the other way around.

Our interest in the current topic was motivated by the question of sharpness for the result obtained in~\cite{FrankLarson}. As a corollary of Theorem~\ref{thm:MainHeat} we obtain that the remainder $o(\lambda^{\gamma+(d-1)/2})$ in the asymptotic expansion in~\eqref{eq: Reisz asymptotics} cannot be improved.
\begin{theorem}\label{thm:MainRiesz}
  Let $R\colon \R_\limplus\to \R$\/ be a non-negative function with $\lim_{\lambda\to \infty}R(\lambda)=0$. There exists an open, bounded, and connected set $\Omega\subset \R^d$ with Lipschitz regular boundary such that for all $\gamma\geq0$
\begin{equation*}
   \limsup_{\lambda\to \infty}\, \frac{\Tr(-\Delta_\Omega-\lambda)^\gamma_\limminus - L_{\gamma, d} |\Omega| \lambda^{\gamma+d/2} + \frac{L_{\gamma, d-1}}{4} \Haus^{d-1}(\partial\Omega) \lambda^{\gamma+(d-1)/2}}{\lambda^{\gamma+(d-1)/2}R(\lambda)}= \infty\, .
\end{equation*}
\end{theorem}

As Theorem~\ref{thm:MainRiesz} concerns the $\limsup$ and not the limit, the result is somewhat weaker than what could be expected from Theorem~\ref{thm:MainHeat}. As such it is reasonable that it should follow from the same principal ideas. However, we are unable to give a direct proof. The main advantage in working with the trace of the heat kernel in comparison to $\Tr(-\Delta_\Omega-\lambda)_\limminus^\gamma$ lies in that we can utilize pointwise estimates for the heat kernel. Corresponding estimates for $\Tr(-\Delta_\Omega-\lambda)_\limminus^\gamma$ are much more delicate. However, it is not very surprising that one can deduce Theorem~\ref{thm:MainRiesz} from Theorem~\ref{thm:MainHeat}. Indeed, by the identity
\begin{equation}\label{eq:Laplace transform}
  \Tr(e^{t\Delta_\Omega}) = \frac{t^{1+\gamma}}{\Gamma(1+\gamma)}\int_0^\infty \Tr(-\Delta_\Omega-\lambda)_\limminus^\gamma e^{-t\lambda}\, d\lambda
\end{equation}
an asymptotic expansion for $\Tr(-\Delta_\Omega-\lambda)_\limminus^\gamma$ as $\lambda\to \infty$ implies a corresponding expansion for $\Tr(e^{t\Delta_\Omega})$ as $t\to 0^\limplus$. 
However, the use of~\eqref{eq:Laplace transform} in our proof of Theorem~\ref{thm:MainRiesz} leads to the $\limsup$ instead of the limit.

\subsection{Additional remarks} In a direction similar to that of Theorems~\ref{thm:MainRiesz} and~\ref{thm:MainHeat} it was shown by Lazutkin and Terman~\cite{lazutkin_estimation_1982} that the error term in~\eqref{eq:Ivrii_asymptotics} cannot be improved on the algebraic scale even among planar convex sets which in addition satisfy the assumptions of Ivrii's result. It is interesting to note that if one considers not the asymptotics of the counting function but of either $\Tr(-\Delta_\Omega-\lambda)^\gamma_\limminus$, with $\gamma \geq 1$, or $\Tr(e^{t\Delta_\Omega})$ one \emph{can} improve the error term on the algebraic scale within convex sets (in any dimension). That this is the case can be deduced from the following uniform inequality proved by the authors in~\cite{FrankLarson}.
  There exists a constant $C>0$ such that for any convex domain $\Omega \subset \R^d$ and all $\lambda \geq 0$,
  \begin{equation}\label{eq:uniform convex}
\begin{aligned}
  \biggl|\Tr(-\Delta_\Omega-\lambda)_\limminus - L_d|\Omega|\lambda^{1+d/2}&+ \frac{L_{d-1}}{4}\Haus^{d-1}(\partial\Omega)\lambda^{1+(d-1)/2}\biggr|\\
   &\leq 
   C \Haus^{d-1}(\partial\Omega)\lambda^{1+(d-1)/2}\bigl(r(\Omega)\sqrt{\lambda}\bigr)^{-1/11}\, 
\end{aligned}
\end{equation}
where $r(\Omega)$ denotes the inradius of $\Omega$.
The corresponding inequality for $\gamma >1$ and $\Tr(e^{t\Delta_\Omega})$ follows from~\eqref{eq:uniform convex} through integration in $\lambda$.

For $\Omega\subset \R^2$ a piecewise smooth domain with finitely many corners it is possible to refine the asymptotic expansions discussed by yet another term~\cite{MR921750, MR0201237, mazzeo_heat_2015, Fedosov_63,MR0217739} (see also~\cite{Fedosov_64} for a similar result in polyhedra in $\R^d$). Let the boundary of $\Omega$ be given  by the union of smooth curve segments $\gamma_j$, $j=1, \ldots, m$, parametrised by arc length and ordered so that $\gamma_j$ meets $\gamma_{j+1}$, and $\gamma_{m}$ meets $\gamma_1$. Let also $\alpha_j\in (0, 2\pi)$ denote the interior angle formed at the point $\gamma_j \cap \gamma_{j+1}$. Then, as $t\to 0^\limplus$, 
\begin{equation}\label{eq:heat trace corner term}
  \Tr(e^{t\Delta_\Omega}) = (4\pi t)^{-1}\biggl(
      |\Omega|- \frac{\sqrt{\pi t}}{2}\Haus^1(\partial\Omega) +\frac{t}{3}\sum_{j=1}^m\biggl[  \int_{\gamma_j}\kappa(s)\, ds + \frac{\pi^2-\alpha_j^2}{2 \alpha_j}\biggr] + o(t)
  \biggr)\, , 
\end{equation}
where $\kappa(s)$ denotes the curvature. It is clear that this expansion cannot extend to the class of Lipschitz sets, as the third term need not be finite.

There has been interesting work on Weyl asymptotics in very irregular sets, specifically sets with fractal boundary. In particular, a lot of work has been directed towards the optimal order of the error term in~\eqref{eq:Weyls law} in this setting. In 1979, it was conjectured by Berry~\cite{Berry1979, Berry1980} that if $\partial\Omega$ has Hausdorff dimension $d-1<d_\Haus\leq d$ then 
\begin{equation*}
    \#\{\lambda_k<\lambda\} = \frac{\omega_d}{(2\pi)^d}|\Omega|\lambda^{d/2}
    + O(\lambda^{d_\Haus/2}) \quad \mbox{as }\lambda\to \infty\, .
\end{equation*}   
However, it was shown by Brossard and Carmona in~\cite{MR834484} that as stated the conjecture of Berry needs to be modified and they suggested that the Hausdorff dimension of $\partial\Omega$ should be replaced by the Minkowski dimension. Subsequently, it was proved by Lapidus~\cite{Lapidus1991} that if $\partial\Omega$ has Minkowski dimension $d_\mathcal{M}$ and finite $d_\mathcal{M}$-dimensional Minkowski content, then the error in the Weyl formula is~$O(\lambda^{d_\mathcal{M}/2})$. We emphasize, however, that the sets considered here, while having non-trivial structure on all scales like fractals, are much more regular.

\section{Proof of Theorem~\ref{thm:MainHeat}}

This section is dedicated to an outline of the proof of Theorem~\ref{thm:MainHeat}. We defer the proof of two ingredients to Sections~\ref{sec:Proof of errorbound} and~\ref{sec: aux lemmas}. Our proof is based on the explicit construction of a Lipschitz domain $\Omega$ satisfying~\eqref{eq:counter_example} for a given function~$g$.

\subsection{Reduction to the two-dimensional case}
In order to simplify the construction we first show that it is sufficient to consider the two-dimensional case. Assume that Theorem~\ref{thm:MainHeat} is known in the case $d=2$. Let $g$ be as in the statement of Theorem~\ref{thm:MainHeat} and fix a Lipschitz regular open, bounded, and connected set $\Omega\subset \R^2$ satisfying~\eqref{eq:counter_example} with $g$ replaced by $\tilde g(t)=\max\{g(t), t^{1/2}\}$. We claim that  $\Omega'=\Omega \times (0, 1)^{d-2}$ satisfies~\eqref{eq:counter_example}. Clearly, $\Omega'$ is open, bounded, connected, and Lipschitz regular.

By the product structure of $\Omega'$, 
\begin{align*}
  \Tr(e^{t\Delta_{\Omega'}}) 
  &= \sum_{l\geq 1}e^{-t\lambda_l(\Omega')}
  =\sum_{j, k\geq 1} e^{-t(\lambda_k(\Omega)+\lambda_j((0, 1)^{d-2}))}
  \\
  & =
  \sum_{j, k\geq 1}e^{-t\lambda_k(\Omega)}e^{-t\lambda_j((0, 1)^{d-2})} =
  \Tr(e^{t\Delta_{\Omega}})\Tr(e^{t\Delta_{(0, 1)^{d-2}}})\, .
\end{align*}
Moreover, 
\begin{equation}\label{eq: Omega' volume}
  |\Omega'|=|\Omega|\quad \mbox{and}\quad \Haus^{d-1}(\partial\Omega') = \Haus^1(\partial\Omega) + 2(d-2)|\Omega|\, .
\end{equation}
It is not difficult to show that
\begin{equation}\label{eq: cube heat trace expansion}
  \Tr(e^{t\Delta_{(0, 1)^{d-2}}}) = (4\pi t)^{-(d-2)/2}\Bigl(1- \frac{\sqrt{\pi t}}{2}2(d-2)+O(t)\Bigr)\, .
\end{equation}
In fact, an explicit computation based on the Poision summation formula and the explicit formulas for the eigenfunctions and eigenvalues on the interval yields: 
\begin{lemma}\label{lem: heat kernel 1D}
   The Dirichlet heat kernel of $-\Delta_{(0, L)}$ evaluated on the diagonal is given by 
   \begin{align}\label{eq:1D kernel on diagonal}
     e^{t\Delta_{(0, L)}}(x, x) 
     &= 
     \frac{1}{\sqrt{4\pi t}}\sum_{m \in \Z}
     \Bigl(
     e^{-\frac{m^2 L^2}{t}}-e^{-\frac{(mL+x)^2}{t}}
     \Bigr)\,.
   \end{align}
\end{lemma}
\begin{proof}[Proof of Lemma~\ref{lem: heat kernel 1D}]
  By scaling it suffices to consider the case $L=1$. Recall that $\lambda_k= \pi^2k^2$ with the corresponding $L^2$-normalized eigenfunction $\phi_k(x)= \sqrt{2}\sin(\pi k x).$ By the explicit formulas for the eigenfunctions and eigenvalues
  \begin{equation*}
    e^{t\Delta_I}(x, x)=
    2\sum_{k=1}^\infty e^{-t \pi^2 k^2}\sin^2(\pi k x)=\sum_{k\in \Z}e^{-t \pi^2 k^2}\sin^2(\pi k x)\,.
  \end{equation*}
  Applying the Poisson summation formula
  \begin{equation*}
    \sum_{k\in \Z}g(k) = \sqrt{2\pi}\sum_{m\in \Z} \check g(2\pi m)
  \end{equation*}
  with $g(\xi)= e^{-t \pi^2 \xi^2}\sin^2(\pi \xi x)$ and using the fact that 
  \begin{equation*}
    \check g(y) = \frac{1}{\sqrt{2\pi}}\int_\R e^{-t \pi^2 \xi^2}\sin^2(\pi \xi x)
    e^{iy\xi}\,d\xi
    = \frac{1}{4\pi \sqrt{2 t}}\Bigl(2e^{-\frac{y^2}{4\pi^2 t}}
   -e^{-\frac{(y-2\pi x)^2}{4\pi^2 t}}-e^{-\frac{(y+2\pi x)^2}{4\pi^2 t}}\Bigr)\,
  \end{equation*}
  completes the proof of the lemma.
\end{proof}

By integrating~\eqref{eq:1D kernel on diagonal} one obtains
\begin{equation*}
  \Tr(e^{t\Delta_{(0, 1)}})= \int_0^1 e^{t\Delta_{(0, 1)}}(x, x)\,dx = (4\pi t)^{-1/2}\Bigl(1- \frac{\sqrt{\pi t}}{2}2+ O(e^{-1/t})\Bigr)\,, \quad \mbox{as } t\to 0^\limplus\,.
\end{equation*}
Thus, by the product structure of $(0, 1)^n$,
 \begin{align*}
   \Tr(e^{t\Delta_{(0, 1)^n}}) = (\Tr(e^{t\Delta_{(0, 1)}}))^n &=\Bigl[(4\pi t)^{-1/2}\Bigl(1- \frac{\sqrt{\pi t}}{2}2+O(e^{-1/t})\Bigr)\Bigr]^n\\
   &= (4\pi t)^{-n/2}\Bigl(1- \frac{\sqrt{\pi t}}{2}2n + O(t)\Bigr)\,,
 \end{align*}
 which is the claimed expansion~\eqref{eq: cube heat trace expansion}.

Combined with $(4\pi t)\Tr(e^{t\Delta_\Omega})=|\Omega|+O(\sqrt{t})$ by Brown~\cite{MR1134755} and~\eqref{eq: Omega' volume}, we find
\begin{align*}
  (4\pi t)^{d/2}\Tr(e^{t\Delta_{\Omega'}})&-|\Omega'|+ \frac{\sqrt{\pi t}}{2}\Haus^{d-1}(\partial\Omega')\\
  &=
  (4\pi t)\Tr(e^{t\Delta_\Omega})-|\Omega|+ \frac{\sqrt{\pi t}}{2}\Haus^1(\partial\Omega) + O(t)\,,
\end{align*}
so
\begin{align*}
  \lim_{t\to 0^\limplus}&\frac{(4\pi t)^{d/2}\Tr(e^{t\Delta_{\Omega'}})-|\Omega'|+ \frac{\sqrt{\pi t}}{2}\Haus^{d-1}(\partial\Omega')}{\sqrt{t}g(t)}\\
  &=
  \lim_{t\to 0^\limplus}
  \frac{(4\pi t)\Tr(e^{t\Delta_{\Omega}})-|\Omega|+ \frac{\sqrt{\pi t}}{2}\Haus^1(\partial\Omega)+O(t)}{\sqrt{t}\tilde g(t)} \frac{\tilde g(t)}{g(t)} =\infty\, , 
\end{align*}
where in the last step we used the choice of $\Omega$ and the fact that by construction $\tilde g(t)\geq g(t)$ and $t^{1/2} \leq \tilde g(t)$. This proves our claim and consequently reduces the proof of Theorem~\ref{thm:MainHeat} to the two-dimensional case.

\subsection{Geometric construction}

We construct a domain $\Omega\subset \R^2$ as follows. The idea is to begin with the square $Q_0=(0, 3)^2$ to which we add triangular teeth to the top edge, all separated from each other and of the same shape but of different sizes and all away from the vertical edges of the large square (see Figure~\ref{fig:counter example}).

\begin{figure}[h]
  \centering
  \begin{tikzpicture}[scale=1.45]

\clip (-0.2,-1.5) rectangle (10.2,1);

\edef \s {1};
\edef \h {0.52};

\draw[thick] (0,0)--(0,-1.2);
\draw[thick] (10,-1.2)--(10,0);
\draw[thick] (0,0)--(\s,0);
\foreach\k in{1,...,4}{
	\draw[thick] (\s, 0)--({\s+0.5*\h}, 0)--({\s+1.5*\h},\h)--({\s+2.5*\h},0);
	\draw[densely dotted] ({\s+0.5*\h},0)--({\s+2.5*\h},0);
	\draw[densely dotted] ({\s+0.5*\h},\h)--({\s+1.5*\h},\h);
	\draw[<->] ({\s+0.5*\h-0.025},0)--({\s+0.5*\h-0.025},\h);
	\node at ({\s+0.5*\h-0.15},0.25) {\scalebox{0.9}{$\tfrac{l_{\scalebox{0.6}{\k}}}{2}$}};
	\xdef\s{\s+2.5*\h};
	\xdef\h{\h/1.2};
}
\foreach\k in{5,...,15}{
	\draw[thick] (\s, 0)--({\s+0.5*\h}, 0)--({\s+1.5*\h},\h)--({\s+2.5*\h},0);
	\draw[densely dotted] ({\s+0.5*\h},0)--({\s+2.5*\h},0);
	\xdef\s{\s+2.5*\h};
	\xdef\h{\h/1.2};
}

\draw[dashed,thick] (\s, 0)--(9,0);
\draw[thick] (9,0)--(10,0);

\end{tikzpicture}
  \caption{A schematic illustration of the top part of the constructed set $\Omega$.}
  \label{fig:counter example}
\end{figure}
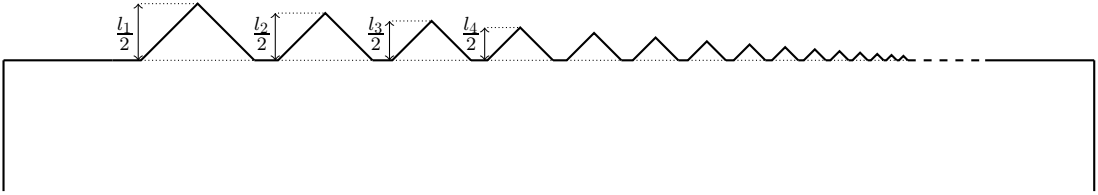

Precisely we consider
\begin{equation}\label{eq: Def Omega}
  \Omega = \{ (x_1, x_2)\in \R^2: 0<x_1<3,\ 0<x_2<H(x_1)\}\,,
\end{equation}
where
\begin{equation*}
  H(x)=3+\sum_{k\geq 1}l_k H_0((x-c_k)/l_k)\, , \quad \mbox{with }
  H_0(x)= 
  \begin{cases}
    x, & x\in [0, 1/2]\\
    1-x, & x\in (1/2, 1]\\
    0, & \mbox{otherwise}\, .  
  \end{cases}
\end{equation*}
Where the sequence $\{l_k\}_{k\geq 1}$ is positive and non-increasing with $\sum_k l_k \leq 1$, and the $c_k$ are chosen increasing and such that $1\leq c_k\leq \min\{c_{k+1}-l_k, 2\}$. This ensure that the supports of the different copies of $H_0$ are disjoint and at least a distance $1$ away from the vertical edges of the square $Q_0$. For instance, the sequence given by $c_1=1$ and $c_k = 1+\sum_{j=1}^{k-1}l_j$ for $k>1$ satisfies all the requirements. Note that at all points $x$ where $H$ is differentiable we have $H'(x)\in \{0, 1, -1\}$. Thus the set $\Omega$ defined by~\eqref{eq: Def Omega} is open, bounded, connected, and has Lipschitz regular boundary.

The idea is that the teeth at a scale much smaller than $\sqrt{t}$ should not have an essential contribution to the trace of the heat kernel. Moreover, the decrease of the area in removing such teeth is small relative to the decrease of length of the boundary. We emphasise that it is not the presence of corners that we are playing with to construct our counterexample. The effect that is essential for our construction is rather that the boundary has non-trivial structure on all scales. Heuristically, the construction should go through if the triangular teeth were replaced by versions where each of the three corners had been smoothed out. However, even though such a modification would make $H_0$ smooth, the function $H$ and the boundary of the set $\Omega$ would be remain merely Lipschitz due to all derivatives of order greater than $1$ blowing up as $l_k$ tends to zero. 

Let $\Omega_M$ be the domain described similarly to $\Omega$ by~\eqref{eq: Def Omega} but with the function $H$ replaced by $H_M$:
\begin{equation}\label{eq: def HM}
  H_M(x)=3+\sum_{\substack{k< M}}l_k H_0((x-c_k)/l_k)\, .
\end{equation}
That is, we remove all teeth after the $M$-th one. We note that $\Omega_{M'}\subset\Omega_M\subset \Omega$, if $M'<M$, $\cup_{M\geq 1}\Omega_M=\Omega$, and by construction
\begin{equation}\label{eq:Area length deficit}
  |\Omega|-|\Omega_M|= \sum_{k\geq M}\frac{l_k^2}{4}\, , 
  \qquad \qquad
  \Haus^{1}(\partial \Omega)-\Haus^{1}(\partial \Omega_M) = \sum_{k\geq M}(\sqrt{2}-1)l_k\, .
\end{equation} 

By monotonicity of Dirichlet eigenvalues under set inclusions
\begin{align*}
(4\pi t)\Tr(e^{t\Delta_\Omega})-&|\Omega|+ \frac{\sqrt{\pi t}}{2}\Haus^{1}(\partial \Omega)\\
&\geq (4\pi t)\Tr(e^{t\Delta_{\Omega_M}})-|\Omega|+ \frac{\sqrt{\pi t}}{2}\Haus^{1}(\partial \Omega)\\
&=
|\Omega_M|-|\Omega|+ \frac{\sqrt{\pi t}}{2}\bigl(\Haus^{1}(\partial\Omega)-\Haus^{1}(\partial \Omega_M)\bigr) + E_M(t)\, , 
\end{align*}
where we define
\begin{equation}\label{eq: EM definition}
  E_M(t)=4\pi t\Tr(e^{t\Delta_{\Omega_M}})-|\Omega_M|+\frac{\sqrt{\pi t}}{2}\Haus^{1}(\partial\Omega_M)\, .
\end{equation}
Our goal is to choose $\{l_k\}_{k\geq 1}$ and $M$ and to bound $E_M(t)$ in such a way that
\begin{equation*}
  \lim_{t\to 0^\limplus}\, \bigl(\sqrt{t}g(t)\bigr)^{-1}\Bigl(|\Omega_M|-|\Omega|+ \frac{\sqrt{\pi t}}{2}\bigl(\Haus^{1}(\partial\Omega)-\Haus^{1}(\partial \Omega_M)\bigr) + E_M(t)\Bigr)= \infty\, .
\end{equation*}
Equivalently, by~\eqref{eq:Area length deficit}, we want to achieve
\begin{equation}\label{eq:asymptotic_goal_2}
  \lim_{t\to 0^\limplus}g(t)^{-1}\biggl(
    \frac{\sqrt{\pi}(\sqrt{2}-1)}{2}\sum_{k\geq M}l_k-\frac{1}{4\sqrt{t}}\sum_{k\geq M}l_k^2+\frac{E_M(t)}{\sqrt{t}}
  \biggr)=\infty\, .
\end{equation}

Before we are able to conclude our proof we shall need to prove some auxiliary results. The first result that we shall need is a bound for $E_M(t)$. This is the content of Proposition~\ref{prop:errorbound} proved in the next section which states that
\begin{equation*}
  E_M(t)\geq -(M+4|\Omega_M|-1)t\, , \quad \mbox{for all }t>0\, .
\end{equation*}
Inserting this bound into~\eqref{eq:asymptotic_goal_2} yields
\begin{equation}\label{eq:final limit}
\begin{aligned}
  g(t)^{-1}\biggl(&
    \frac{\sqrt{\pi}(\sqrt{2}-1)}{2}\sum_{k\geq M}l_k-\frac{1}{4\sqrt{t}}\sum_{k\geq M}l_k^2+\frac{E_M(t)}{\sqrt{t}}
  \biggr)\\
  &\geq
   g(t)^{-1}\biggl(
    \frac{\sqrt{\pi}(\sqrt{2}-1)}{2}\sum_{k\geq M}l_k-\frac{1}{4\sqrt{t}}\sum_{k\geq M}l_k^2-(M+4|\Omega_M|-1)\sqrt{t}
  \biggr)\,.
\end{aligned}
\end{equation}
We claim that one can choose a sequence $l_k$ and a decreasing function $M(t)$ in such a manner that the quantity~\eqref{eq:final limit} tends to infinity as $t\to 0^\limplus$. Indeed, setting $\tau = \sqrt{t}$, and defining $G(\tau)=g(\tau^2)$ the existence of such a sequence follows from Corollary~\ref{cor: Aux2 sequences} proved in Section~\ref{sec: aux lemmas} below. 
Note that for any $l_k$ and $M$ provided by Corollary~\ref{cor: Aux2 sequences} one has necessarily $\lim_{t\to 0^\limplus}M(t)=\infty$ and therefore $(4|\Omega_M|-1)\sqrt{t}=o(M\sqrt{t})$. Thus for the $l_k$ and $M$ provided by Corollary~\ref{cor: Aux2 sequences} only the first term in~\eqref{eq:final limit} affects the limit. Contingent on us proving Proposition~\ref{prop:errorbound} and Corollary~\ref{cor: Aux2 sequences}, this completes the proof of Theorem~\ref{thm:MainHeat}.\qed

\section{An estimate for the error term}
\label{sec:Proof of errorbound}

Recall that
\begin{equation*}
  \Omega_M = \{(x_1, x_2)\in \R^2: 0<x_1<3,\ 0<x_2< H_M(x_1)\}\,
\end{equation*}
with $H_M$ given by~\eqref{eq: def HM}, and $E_M(t)$ is defined in~\eqref{eq: EM definition}.
Our goal in this section is to prove the following proposition:
\begin{proposition}\label{prop:errorbound}
  Let $\{l_k\}_{k\geq 1}$ be a non-negative sequence with $\sum_{k\geq 1}l_k \leq1$. For every $M\geq 1$ and $t>0$, 
  \begin{equation*}
  E_{M}(t)\geq - (M+ 4|\Omega_M|-1) t\, .
  \end{equation*}
\end{proposition}
\begin{remark}
    By comparing the bound in Proposition~\ref{prop:errorbound} with the third term in~\eqref{eq:heat trace corner term} the linear dependence on $M$ and $t$ in Proposition~\ref{prop:errorbound} appears to be order-sharp. Indeed, for each fixed $M$ the three-term expansion~\eqref{eq:heat trace corner term} states that
    \begin{equation*}
      E_M(t)= 
      \frac{t}{3}\sum_{k<M}\biggl[\frac{\pi^2-(\pi/2)^2}{2(\pi/2)}+2\frac{\pi^2-(5\pi/4)^2}{2(5\pi/4)} \biggr]+ o(t)
      =
  \frac{\pi}{10}(M-1)t+ o(t)\, .
    \end{equation*}
\end{remark}

Let $H_\Omega(x, t)=e^{t\Delta_\Omega}(x, x)$ denote the heat kernel of the Dirichlet Laplacian on $\Omega$ evaluated on the diagonal. For our proof of Proposition~\ref{prop:errorbound} we need a pointwise lower bound for $H_{(0, L)^2}(x, t)$ which has the asymptotically correct behaviour close to the boundary. 
\begin{lemma}\label{lem:1D kernel bound}
  For the heat kernel of Dirichlet Laplacian on the interval $(0, L)$, 
  \begin{equation*}
    H_{(0, L)}(x, t) \geq (4\pi t)^{-1/2}\bigl(1-e^{-x^2/t}-e^{-(L-x)^2/t}\bigr)\, .
  \end{equation*}
\end{lemma}
\begin{proof}[Proof of Lemma~\ref{lem:1D kernel bound}]
  By Lemma~\ref{lem: heat kernel 1D},
  \begin{align*}
    H_{(0, L)}(x, t)
    &= 
    \frac{1}{2\sqrt{\pi t}}\sum_{m\in \Z}\Bigl(e^{-\frac{m^2L^2}{t}}-e^{-\frac{(mL+x)^2}{t}}\Bigr)\\
    &=
    \frac{1}{2\sqrt{\pi t}}\biggl[1-e^{-\frac{x^2}{t}}-e^{-\frac{(L-x)^2}{t}}+\sum_{m\geq 1}\Bigl(2e^{-\frac{m^2L^2}{t}}-e^{-\frac{(mL+x)^2}{t}}-e^{-\frac{(mL+(L-x))^2}{t}}\Bigr)\biggr]\\
    &\geq
    \frac{1}{2\sqrt{\pi t}}\Bigl[1-e^{-\frac{x^2}{t}}-e^{-\frac{(L-x)^2}{t}}\Bigr]\,,
  \end{align*}
  where in the final step we used $e^{-\frac{m^2L^2}{t}}-e^{-\frac{(mL+x)^2}{t}} \geq 0$ and $e^{-\frac{m^2L^2}{t}}-e^{-\frac{(m L+(L-x))^2}{t}}\geq 0$, for $m\geq 1$ and $x \in (0, L)$.
\end{proof}

\begin{corollary}\label{cor:Heat kernel square}
For the heat kernel of Dirichlet Laplacian on the square $(0, L)^2$ we have
  \begin{align*}
    H_{(0, L)^2}(x, t) &\geq (4\pi t)^{-1}
    \bigl(1- e^{-d(x_1)^2/t}-e^{-d(x_2)^2/t}-4e^{-L^2/(4t)}\bigr)\, , 
  \end{align*}
where $x=(x_1, x_2)\in (0, L)^2$ and $d(x_j)= \min\{x_j, L-x_j\}$.
\end{corollary}
\begin{proof}[Proof of Corollary~\ref{cor:Heat kernel square}]
  By symmetry it suffices to consider the region where $0<x_1, x_2 \leq L/2$. By Lemma~\ref{lem:1D kernel bound} and the product structure of the heat kernel in $(0, L)^2$ we have
  \begin{equation}\label{eq: product bound}
   (4\pi t) H_{(0, L)^2}(x, t) 
   \geq \bigl(1- e^{-x_1^2/t}-e^{-(L-x_1)^2/t}\bigr)_\limplus\bigl(1-e^{-x_2^2/t}-e^{-(L-x_2)^2/t}\bigr)_\limplus\, .
   \end{equation}
   Writing $A_\limplus B_\limplus = (A+A_\limminus)(B+B_\limminus)$ we find
  \begin{align*}
  \bigl(1- e^{-x_1^2/t}-e^{-(L-x_1)^2/t}\bigl)_\limplus&\bigl(1-e^{-x_2^2/t}-e^{-(L-x_2)^2/t}\bigr)_\limplus\\
   &=
   1- e^{-x_1^2/t}-e^{-x_2^2/t}-e^{-(L-x_1)^2/t}-e^{-(L-x_2)^2/t}\\
   &\quad  + e^{-(x_1^2+x_2^2)/t}+ e^{-(x_1^2+(L-x_2)^2)/t)}+e^{-((L-x_1)^2+x_2^2)/t}\\
   &\quad  + e^{-((L-x_1)^2+(L-x^2)^2)/t}\\
   &\quad - \bigl(1- e^{-x_1^2/t}-e^{-(L-x_1)^2/t}\bigr)\bigl(1-e^{-x_2^2/t}-e^{-(L-x_2)^2/t}\bigr)_\limminus\\
   &\quad - \bigl(1-e^{-x_2^2/t}-e^{-(L-x_2)^2/t}\bigr)\bigl(1- e^{-x_1^2/t}-e^{-(L-x_1)^2/t}\bigr)_\limminus\\
   &\quad+ \bigl(1- e^{-x_1^2/t}-e^{-(L-x_1)^2/t}\bigr)_\limminus\bigl(1-e^{-x_2^2/t}-e^{-(L-x_2)^2/t}\bigr)_\limminus\, .
   \end{align*}
   By discarding positive terms, using $L-x_j\geq L/2$ and, since $1-e^{x_j^2/t}\geq 0$, 
   \begin{equation*}
    \bigl(1- e^{-x_j^2/t}-e^{-(L-x_j)^2/t}\bigr)_\limminus \leq e^{-(L-x_j)^2/t}\, , 
  \end{equation*}
  we find 
   \begin{align*}
   \bigl(1- e^{-x_1^2/t}-e^{-(L-x_1)^2/t}\bigr)_\limplus&\bigl(1-e^{-x_2^2/t}-e^{-(L-x_2)^2/t}\bigr)_\limplus\\
   &\geq
   1- e^{-x_1^2/t}-e^{-x_2^2/t}-2e^{-L^2/(4t)}\\
   &\quad - \bigl(1-e^{-x_2^2/t}-e^{-(L-x_2)^2/t}\bigr)_\limminus - \bigl(1- e^{-x_1^2/t}-e^{-(L-x_1)^2/t}\bigr)_\limminus\\
   &\geq
   1- e^{-x_1^2/t}-e^{-x_2^2/t}-4e^{-L^2/(4t)}\, .
  \end{align*}
  Which when combined with~\eqref{eq: product bound} yields the claimed bound and completes the proof of the corollary.
\end{proof}

We are now ready to prove Proposition~\ref{prop:errorbound}.

\begin{proof}[Proof of Proposition~\ref{prop:errorbound}]
Our proof is based on the fact that if $x\in \Omega\subset \Omega'$ then $H_\Omega(x, t)\leq H_{\Omega'}(x, t) $ for any $t>0$~\cite{MR990239}. We will bound $\Tr(e^{t\Delta_{\Omega_M}})$ from below by bounding the heat kernel pointwise from below in terms of the heat kernel of appropriate squares contained in $\Omega_M$.

At several points in the proof we shall use the fact that, for $\delta, t >0$, 
\begin{equation}\label{eq: erf estimates}
  0\leq \frac{\sqrt{\pi t}}{2} - \int_0^\delta e^{-s^2/t}\, ds = \int_\delta^\infty e^{-s^2/t}\, dt = \frac{\sqrt{t}}{2}\int_{\delta^2/t}^\infty \frac{e^{-y}}{\sqrt{y}}\, dy \leq \frac{t}{2\delta}e^{-\delta^2/t}\, .
\end{equation}

We introduce some notation for different regions in $\Omega$. Let $Q_k$ be the square of side-length $\sqrt{2}$ placed so that one corner matches the $k$-th tooth. By construction, each tooth is at least a distance $1$ away from the vertical sides of $Q_0=(0, 3)^2$. Therefore, $Q_k \subset \Omega_M$ for each $k<M$. Let $M_k$ be the triangular region corresponding to the $k$-th tooth and let $L_k$ denote the same region but mirrored across the boundary of $Q_0$. The set $\overline{L_k\cup M_k}$ is a square with sidelength $l_k/\sqrt{2}$. Since $\frac{l_k}{\sqrt{2}} < \sum_{j\geq 1}\frac{l_j}{\sqrt{2}} \leq \frac{1}{\sqrt{2}} = \frac{\sqrt{2}}{2}$ the set $L_k\cup M_k$ is contained in one of the quarters of $Q_k$ obtained by cutting parallel to its sides. See Figure~\ref{fig:MkLk} for an illustration of what was described above.

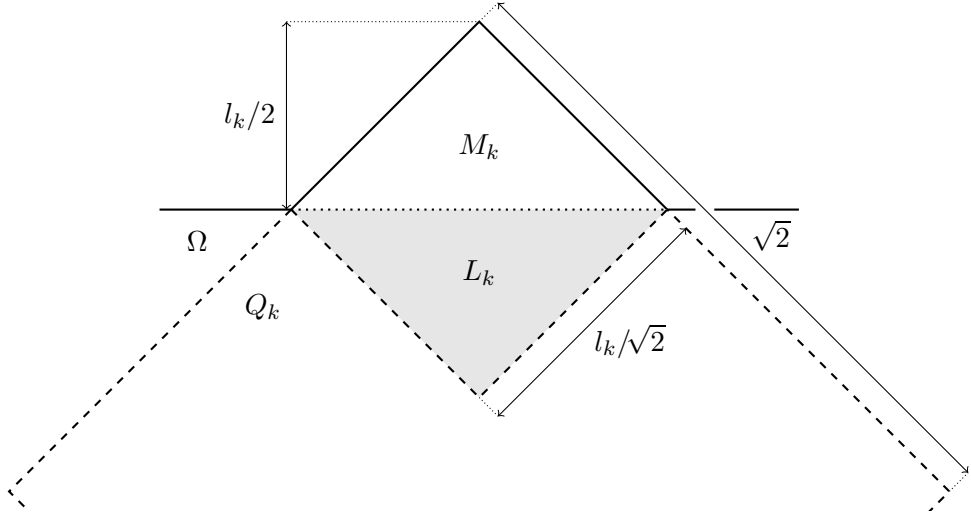
\begin{figure}[h]
  \centering
  \begin{tikzpicture}[scale=2.5]

\clip (-2.7,-1.6) rectangle (2.7,1.2);

	\draw[thick] (-1.7, 0)--(-1,0)--(0,1)--(1,0)--(1.15,0);
	\draw[thick] (1.7,0)--(1.25,0);
	\fill[gray!20] (-1,0)--(1,0)--(0,-1);
	\draw[thick,dashed] (-1,0)--(0,-1)--(1,0);
	\draw[thick,dotted] (-1,0)--(1,0);

	\draw[<->] (-1.025,0)--(-1.025,1);
	\draw[densely dotted] (-1.025,1)--(0,1);
	\draw[dashed, thick] (-1,0)--(-2.5,-1.5)--(0,-4)--(2.5,-1.5)--(1,0);

	\node at (-1.5, -1/6) {$\Omega$};
	\node at (0,1/3) {$M_k$};
	\node at (0,-1/3) {$L_k$};
	\node at (-1.22, 1/2) {$l_k/2$};
	\node at (1.55,-0.14) {$\sqrt{2}$};

	\draw[densely dotted] (0,1)--(0.1,1.1);
	\draw[densely dotted] (2.5,-1.5)--(2.6,-1.4);
	\draw[<->] (2.6,-1.4)--(0.1,1.1);

	\node at ({-1.5+0.35}, {-1/6-0.35}) {$Q_k$};

	\draw[densely dotted] (0,-1)--(0.1,-1.1);
	\draw[<->] (0.1,-1.1)--(1.1,-0.1);
	\node at ({1/2+0.3,-1/2-0.2}) {$l_k/{\!\sqrt{2}}$};

\end{tikzpicture}
  \caption{Depiction of the regions $L_k$, $M_k$, and the corresponding square $Q_k$ contained in $\Omega$.}
  \label{fig:MkLk}
\end{figure}

Writing $\Tr(e^{t\Delta_{\Omega_M}})$ as the integral of $H_{\Omega_M}$ we can split the integral into pieces to be treated separately, 
\begin{align*}
   \Tr(e^{t\Delta_{\Omega_M}}) 
   &= 
   \int_{\Omega_M}H_{\Omega_M}(x, t)\, dx\\
   &= 
   \int_{Q_0\setminus \cup_k L_k} H_{\Omega_M}(x, t)\, dx  
   + \sum_{k< M} \int_{L_k\cup M_k}H_{\Omega_M}(x, t)\, dx\, .
\end{align*}
For the first integral we use $H_{\Omega_M}(x, t)\geq H_{Q_0}(x, t)$ for all $x\in Q_0$ and $t> 0$. For the integral over $L_k\cup M_k$ we use the fact that $H_{\Omega_M}(x, t)$, for $x\in L_k\cup M_k$ and $t>0$, is bounded from below by $H_{Q_k}(x, t)$. What one finds is
\begin{equation}\label{eq: Pointwise Lower bound OmegaM}
 \Tr(e^{t\Delta_{\Omega_M}}) 
   \geq \int_{Q_0\setminus \cup_{k<M} L_k} H_{Q_0}(x, t)\, dx  + \sum_{k< M} \int_{L_k\cup M_k}H_{Q_k}(x, t)\, dx\, .
\end{equation}

To bound the integral of $H_{Q_0}$ we apply Corollary~\ref{cor:Heat kernel square} which yields, with $d(x)=\min\{x, 3-x\}$,
\begin{align*}
\int_{Q_0\setminus \cup_{k<M} L_k}H_{Q_0}(x, t)\, dx
& \geq
   (4\pi t)^{-1}\int_{Q_0\setminus \cup_{k<M} L_k} \bigl(1-e^{-d(x_1)^2/t}-e^{-d(x_2)^2/t}-4e^{-9/(4t)}\bigr)\, dx  \\
   & =
   (4\pi t)^{-1}|Q_0\setminus \cup_{k<M}L_k|(1-4e^{-9/(4t)})  \\
   &\quad-(4\pi t)^{-1}\int_{Q_0} \bigl(e^{-d(x_1)^2/t}+e^{-d(x_2)^2/t}\bigr)\, dx \\
   &\quad + (4\pi t)^{-1}\sum_{k<M}\int_{L_k} \bigl(e^{-d(x_1)^2/t}+e^{-d(x_2)^2/t}\bigr)\, dx\, .
 \end{align*} 
By symmetry and~\eqref{eq: erf estimates},
\begin{equation*}
  \int_{Q_0} \bigl(e^{-d(x_1)^2/t}+e^{-d(x_2)^2/t}\bigr)\, dx
  =  12\int_0^{3/2}e^{-x_1^2/t}\, dx_1
  \leq 6\sqrt{\pi t} = \frac{\sqrt{\pi t}}{2}\Haus^{1}(\partial Q_0)\, , 
\end{equation*}
and similarly, for any $k\geq 1$, 
\begin{align*}
  \int_{L_k} \bigl(e^{-d(x_1)^2/t}+e^{-d(x_2)^2/t}\bigr)\, dx 
  &\geq
  \int_{L_k}e^{-d(x_2)^2/t}\, dx\\
  &=
  \int_0^{l_k/2} (l_k-2s)e^{-s^2/t}\, ds\\
  &= l_k \int_0^{l_k/2}e^{-s^2/t}\, ds - t(1-e^{-l_k^2/(4t)})\\
  &\geq  \frac{\sqrt{\pi t}}{2}l_k -t \\
  &= \frac{\sqrt{\pi t}}{2}\Haus^{1}(\partial L_k \cap \partial Q_0)-t\, .
\end{align*}
Thus, we have shown that 
\begin{equation}\label{eq: bound Q minus Lk}
\begin{aligned}
  \int_{Q_0\setminus \cup_{k<M}L_k}H_{Q_0}(x, t)\,dx 
  &\geq 
  (4\pi t)^{-1}\biggl[|Q_0\setminus \cup_{k<M}L_k|(1-4e^{-9/(4t)})\\
  &\quad - \frac{\sqrt{\pi t}}{2}\Bigl(\Haus^1(\partial Q_0) - \sum_{k<M}\Haus^1(\partial L_k \cap \partial Q_0)\Bigr)- (M-1)t\biggr]\,.
\end{aligned}
\end{equation}

 What remains of our proof is to estimate the second integral in~\eqref{eq: Pointwise Lower bound OmegaM}.
 By Corollary~\ref{cor:Heat kernel square}, and with $d(x_j)$ interpreted appropriately,
\begin{align*}
    \int_{L_k\cup M_k}&H_{Q_k}(x, t)\, dx \\
    &\geq
    (4\pi t)^{-1}\int_{L_k\cup M_k} \bigl(1-e^{-d(x_1)^2/t}-e^{-d(x_2)^2/t}-4e^{-1/(2t)}\bigr)\, dx \\
    &=
    (4\pi t)^{-1}\biggl[|L_k\cup M_k|(1-4e^{-1/(2t)}) -\int_{L_k\cup M_k} \bigl(e^{-d(x_1)^2/t}+e^{-d(x_2)^2/t}\bigr)\, dx\biggr]\, .
\end{align*} 
By symmetry, ~\eqref{eq: erf estimates}, and the fact that $L_k\cup M_k$ is contained in a quarter of the larger square $Q_k$ we have that
\begin{align*}
  \int_{L_k\cup M_k} \bigl(e^{-d(x_1)^2/t}+e^{-d(x_2)^2/t}\bigr)\, dx
  &=
  2 \int_{L_k\cup M_k} e^{-d(x_1)^2/t}\, dx\\
  & = \sqrt{2}l_k \int_{0}^{l_k/\sqrt{2}} e^{-s^2/t}\, ds\\
  &\leq
  \frac{\sqrt{\pi t}}{2}\sqrt{2}l_k\\
  &= \frac{\sqrt{\pi t}}{2}\Haus^{1}(\partial \Omega_M \cap \partial M_k) \, .
\end{align*}
Thus, we have shown that
\begin{equation}\label{eq: bound Lk cup Mk}
  \int_{L_k \cup M_k}H_{Q_k}(x, t)\,dx 
  \geq 
  (4\pi t)^{-1}
  \biggl[|L_k\cup M_k|(1-4e^{-1/(2t)})- \frac{\sqrt{\pi t}}{2}\Haus^1(\partial \Omega_M \cap \partial M_k)\biggr]\,.
\end{equation}

Combining~\eqref{eq: bound Q minus Lk} and~\eqref{eq: bound Lk cup Mk} we have arrived at the bound
\begin{equation*}
  \begin{aligned}
    \Tr(e^{t\Delta_{\Omega_M}}) 
    &\geq (4\pi t)^{-1}\biggl[
      |\Omega_M| - \frac{\sqrt{\pi t}}{2}\Haus^{1}(\partial\Omega_M)
         - (M-1)t- 4|\Omega_M|e^{-1/(2t)}
    \biggr]\, .
  \end{aligned}
\end{equation*}
Rearranging, and since $e^{-1/(2t)}\leq t$, for all $t\geq 0$, we have proved that
\begin{equation*}
   E_M(t) \geq -(M+4|\Omega_M|-1)t\, , 
 \end{equation*} 
which completes the proof of Proposition~\ref{prop:errorbound}.
\end{proof}

\section{Auxiliary results}
\label{sec: aux lemmas}

In this section we prove a number of technical results which will be needed in the proof of Theorem~\ref{thm:MainHeat}. Specifically we prove that one can find a function whose $L^1$ tail tends to zero slower than a given non-negative function $G$, while satisfying some additional properties.

\begin{proposition}\label{prop: Aux1 integrals}
  Let $G\colon (0, 1)\to \R$ be a non-negative function with $\lim_{\tau \to 0^\limplus}G(\tau )=0$. There exists a strictly decreasing non-negative and smooth function $h\in L^1((0, 1))$ such that
  \begin{equation*}
    \lim_{\tau \to 0^\limplus} \frac{\int_0^\tau  h(x)\, dx}{G(\tau )} =\infty\, , \quad \mbox{and} \quad \lim_{\tau\to 0^\limplus}\frac{\tau h(\tau )+\frac{1}{\tau }\int_0^\tau  x h(x)\, dx}{\int_0^\tau  h(x)\, dx} =0\, .
  \end{equation*}
\end{proposition}

As a consequence of Proposition~\ref{prop: Aux1 integrals} we can prove the following result in the setting of sequences which provides the final ingredient to complete the proof of Theorem~\ref{thm:MainHeat}.
\begin{corollary}\label{cor: Aux2 sequences}
  Let $G\colon (0, 1)\to \R$ be non-negative function with $\lim_{\tau\to 0^\limplus}G(\tau)=0$ and let $A>0$. There exists a positive non-increasing sequence $\{l_k\}_{k\geq 1}$ with $\sum_{k\geq 1}l_k=A$ and a decreasing function $M\colon (0, 1) \to \N$ such that
  \begin{equation*}
    \lim_{\tau\to 0^\limplus} \frac{\sum_{k\geq M(\tau)}l_k}{G(\tau)}=\infty\, , 
    \quad \mbox{and}\quad 
    \lim_{\tau\to 0^\limplus}\frac{\tau M(\tau) +\frac{1}{\tau}\sum_{k\geq M(\tau)} l_k^2}{\sum_{k\geq M(\tau)}l_k}=0\, .
  \end{equation*}
\end{corollary}
We first prove Proposition~\ref{prop: Aux1 integrals} and then show how to deduce Corollary~\ref{cor: Aux2 sequences} from it. To simplify the proof of Proposition~\ref{prop: Aux1 integrals} it will be convenient to first show that we may assume that $G$ is fairly well-behaved.

\begin{lemma}\label{lem:Regular Modulus of continuity}
  Let $G\colon (0, 1)\to \R$ be a non-negative function with\/ $\lim_{\tau\to 0^\limplus}G(\tau)=0$. There exists a function\/ $\widehat G\in C^\infty((0, 1))$ such that
  \begin{enumerate}[label=(\alph*)]
    \item\label{itm1} $\lim_{\tau\to 0^\limplus}\widehat G(\tau)=0$, 
    \item\label{itm2} $\widehat G(\tau)>G(\tau)$ for all $\tau\in (0, 1)$, 
    \item\label{itm3} $\widehat G'(\tau)>0$ for all $\tau\in (0, 1)$, and
    \item\label{itm4}  $\widehat G''(\tau)<0$ for all $\tau\in (0, 1)$.
  \end{enumerate}
\end{lemma}

\begin{proof}[Proof of Lemma~\ref{lem:Regular Modulus of continuity}]
  Fix $G$ as in the lemma. Define, for $\tau >0$, 
  \begin{equation*}
    \bar G(\tau)= \inf\{a\tau + b : a\geq 0, b\geq 0, G(s)\leq a s+b \mbox{ for all } s\in (0, 1)\}\, .
  \end{equation*}
  It is clear from the construction that $\bar G$ is non-decreasing, concave, and satisfies $G(\tau)\leq \bar G(\tau)$ for all $\tau\in (0, 1)$. That $\lim_{\tau\to 0^\limplus} \bar G(\tau)=0$ follows if we can prove that for every $b>0$ there exists an $a\geq 0$ such that $G(s)\leq a s + b$ for all $s\in [0, 1]$. Since $G(s)-b$ is negative for $s$ small enough, the choice $a=\sup_{0<s<1}\frac{(G(s)-b)_\limplus}{s}$ works.

  Fix $\phi\in C_0^\infty((1, 2))$ non-negative with $\int_\R \phi(x)\, dx =1$. Let
  \begin{equation*}
     \tilde G(\tau) = \int_\R \phi(x)\bar G(\tau x)\, dx= \frac{1}{\tau}\int_\R \phi(y/\tau)\bar G(y)\, dy\, .
  \end{equation*}
  Since $\phi\in C_0^\infty((1, 2))$ it holds that $\tilde G\in C^\infty((0, 1))$. Since $\phi \geq 0$, $\supp \phi \subseteq [1, 2]$, and $\bar G$ is non-decreasing, 
  \begin{equation}\label{eq: modified G is larger}
    \tilde G(\tau) =\int_\R \phi(x)\bar G(\tau x)\, dx \geq \bar G(\tau)\geq G(\tau)\, , 
  \end{equation}
  and similarly
  \begin{equation}\label{eq: modified G tends to zero}
    \lim_{\tau\to 0^\limplus} \tilde G(\tau) = \lim_{\tau\to 0^\limplus}\int_\R \phi(x)\bar G(\tau x)\, dx \leq \lim_{\tau\to 0^\limplus} \bar G(2\tau)=0\, .
  \end{equation}
  Furthermore, $\tilde G$ is increasing and concave, both of which are consequences of the corresponding properties for $\bar G$. Indeed, for $0<\tau_1<\tau_2\leq 1$, 
  \begin{equation}\label{eq: modified G is increasing}
    \tilde G(\tau_2)-\tilde G(\tau_1) = \int_\R \phi(x)\bigl(\bar G(\tau_2x)-\bar G(\tau_1x)\bigr)\, dx \geq 0, 
  \end{equation}
  and, for any $\alpha \in (0, 1)$, 
  \begin{equation}\label{eq: modified G is concave}
  \begin{aligned}
    \tilde G((1-\alpha)\tau_1+\alpha \tau_2) 
    &= 
    \int_\R \phi(x)\bigl(\bar G((1-\alpha)\tau_1x+\alpha \tau_2x)\bigr)\, dx\\
    &\geq
    \int_\R\phi(x)\bigl((1-\alpha)\bar G(\tau_1x)+\alpha\bar G(\tau_2x)\bigr)\, dx\\
    &=
    (1-\alpha) \tilde G(\tau_1)+ \alpha \tilde G(\tau_2)\, .
  \end{aligned}
  \end{equation}
  From~\eqref{eq: modified G is larger}, ~\eqref{eq: modified G tends to zero}, ~\eqref{eq: modified G is increasing}, and~\eqref{eq: modified G is concave} it follows that $\widehat G(\tau) =\tilde G(\tau)+\sqrt{\tau}$ satisfies the properties claimed in the lemma.
\end{proof}

\begin{proof}[Proof of Proposition~\ref{prop: Aux1 integrals}]
By Lemma~\ref{lem:Regular Modulus of continuity} we can without loss of generality assume that $G$ is smooth, strictly increasing, and concave.

Set $h(\tau) = G'(\tau)U(G(\tau))$ for some non-negative, integrable, and strictly decreasing $U\in C^\infty((0, \|G\|_\infty))$ to be specified. By the assumptions on $G, U$ we have $h\in C^\infty$ and, for $0<\tau_1<\tau_2\leq 1$, 
\begin{equation*}
  h(\tau_1)-h(\tau_2) =  G'(\tau_1)U(G(\tau_1))- G'(\tau_2)U(G(\tau_2)) >
  (G'(\tau_1)- G'(\tau_2))U(G(\tau_1)) \geq 0\, , 
\end{equation*}
thus $h$ is strictly decreasing. Moreover, $h$ is integrable since
\begin{equation*}
  \int_0^1 h(\tau)\, d\tau = \int_0^1 G'(\tau)U(G(\tau))\, d\tau = \int_0^{G(1)}U(y)\, dy <\infty\, .
\end{equation*}
By the same computation we have that
\begin{equation}\label{eq: main term equality}
  \int_0^\tau h(x)\, dx = \int_0^{G(\tau)}U(y)\, dy\, , 
\end{equation}
as such we see that in order to satisfy the first claim of the proposition we simply need to choose $U$ so that $\int_0^s U(y)\, dy \gg s$.  

Since $G$ is increasing and concave we can bound
\begin{equation}\label{eq: concavity bound}
  G(\tau) = \int_0^\tau G'(x)\, dx \geq \tau G'(\tau)\, .
\end{equation}
Using this we wish to estimate terms in the numerator of the second limit.

Since $U\geq 0$ the inequality~\eqref{eq: concavity bound} implies that
\begin{equation}\label{eq: th(t) term bound}
  \tau h(\tau) = \tau G'(\tau)U(G(\tau)) \leq 
  G(\tau)U(G(\tau))\, .
\end{equation}

For the second term an integration by parts yields
\begin{align*}
  \frac{1}{\tau}\int_0^\tau x h(x)\, dx 
  &=
  \frac{1}{\tau}\int_0^\tau x G'(x)U(G(x))\, dx \\
  &=
  - \frac{1}{\tau}\int_0^\tau \Bigl[G(x)U(G(x))+xG(x)G'(x)U'(G(x))\Bigr]\, dx + G(\tau)U(G(\tau))\, , 
\end{align*}
where we used $\lim_{x \to 0^\limplus} x G(x)U(G(x))=0$ by integrability and monotonicity of $U$. Using again~\eqref{eq: concavity bound} to bound the first term in the brackets one finds
\begin{align}
  \frac{1}{\tau}\int_0^\tau x h(x)\, dx 
  &\leq 
  - \frac{1}{\tau}\int_0^\tau \Bigl[x G'(x)U(G(x))+xG(x)G'(x)U'(G(x))\Bigr]\, dx + G(t)U(G(\tau))\nonumber \\
  &=
  - \frac{1}{\tau}\int_0^\tau xG'(x)\Bigl[U(G(x))+G(x)U'(G(x))\Bigr]\, dx + G(\tau)U(G(\tau))\, \label{eq: differential inequality}.
\end{align}

Let $U(y)=y^{-1}\bigl(\log\frac{e^{2}\|G\|_\infty}{y}\bigr)^{-2}$. Note that, for $y\in (0, \|G\|_\infty)$, 
\begin{equation*}
   U'(y) 
   = -\frac{U(y)}{y}\Bigl(1-2\Bigl(\log\frac{e^{2}\|G\|_\infty}{y}\Bigr)^{-1}\Bigr)< 0 \quad \mbox{for all } 0<y < \|G\|_\infty\, .
\end{equation*} 
Moreover, 
\begin{equation}\label{eq: U properties} 
\begin{aligned}
  \int_0^{G(\tau)}U(y)\, dy &= \Bigl(\log\frac{e^2\|G\|_\infty}{G(\tau)}\Bigr)^{-1}\, , \\
  G(\tau)U(G(\tau)) &= \Bigl(\log\frac{e^2\|G\|_\infty}{G(\tau)}\Bigr)^{-2}\, , \ \mbox{and }\\
  U(y)+y U'(y) &= 2 y^{-1}\Bigl(\log\frac{e^2\|G\|_\infty}{y}\Bigr)^{-3} > 0\quad \mbox{for all } 0<y<\|G\|_\infty\, .
\end{aligned}
\end{equation}
Thus, by~\eqref{eq: differential inequality} and~\eqref{eq: th(t) term bound},
\begin{equation}\label{eq: lower order integrals bound}
  \tau h(\tau)+ \frac{1}{\tau}\int_0^\tau xh(x)\,dx \leq \tau h(\tau)+ G(\tau)U(G(\tau)) \leq 2G(\tau)U(G(\tau))\,.
\end{equation}

Combining, \eqref{eq: main term equality} and \eqref{eq: th(t) term bound} with the equations in~\eqref{eq: U properties} we finally obtain that
\begin{equation*}
  \lim_{\tau\to 0^\limplus} \frac{\int_0^\tau h(x)\, dx}{G(\tau)} = \lim_{\tau\to 0^\limplus} G(\tau)^{-1}\Bigl(\log\frac{e^2\|G\|_\infty}{G(\tau)}\Bigr)^{-1} = \infty\, , 
\end{equation*}
and similarly, by~\eqref{eq: lower order integrals bound},
\begin{equation*}
  0\leq \lim_{\tau\to 0^\limplus} \frac{\tau h(\tau)+\frac{1}{\tau}\int_0^\tau x h(x)\, dx}{\int_0^\tau h(x)\, dx} \leq \lim_{\tau\to 0^\limplus}  2\Bigl(\log\frac{e^2\|G\|_\infty}{G(\tau)}\Bigr)^{-1} = 0\, , 
\end{equation*}
which concludes the proof of the proposition.
\end{proof}

\begin{proof}[Proof of Corollary~\ref{cor: Aux2 sequences}]
  Without loss of generality we can assume that $G(\tau)\geq \tau$. By Proposition~\ref{prop: Aux1 integrals} there exists a non-negative strictly decreasing smooth function $h\in L^1((0, 1))$ such that
  \begin{equation*}
    \lim_{\tau\to 0^\limplus} \frac{\int_0^\tau h(x)\, dx}{G(\tau)} =\infty\, , \quad \mbox{and} \quad \lim_{\tau\to 0^\limplus}\frac{\tau h(\tau)+\frac{1}{\tau}\int_0^\tau x h(x)\, dx}{\int_0^\tau h(x)\, dx} =0\, .
  \end{equation*}
  The assumption that $G(\tau)\geq \tau$ implies that $\lim_{\tau\to 0^\limplus}h(\tau)=\infty$.

  Since $h$ is strictly decreasing we can consider its inverse, $h^{-1}\colon (h(1), \infty) \to \R$, which is again strictly decreasing. Define $f\colon (0, \infty) \to \R$ by $f(y)=h^{-1}(y)\1_{y> h(1)}+ \1_{y\leq h(1)}$. Then
  \begin{equation*}
    \int_0^\infty f(y)\, dy = h(1)+\int_{h(1)}^\infty h^{-1}(y)\, dy = h(1)-\int_0^1 x h'(x)\, dx = \int_0^1 h(x)\, dx\, , 
  \end{equation*}
  where we used the fact that $h$ is integrable and monotone to conclude that $\lim_{\tau\to 0^\limplus}\tau h(\tau)=0$. Since the equation $f(y)=\tau$ has a unique solution for each $\tau \in (0, 1)$, the inverse of $f^{-1}\colon (0, 1)\to \R$ is well-defined and by construction $f^{-1}(\tau)=h(\tau)$.

  Since $f$ is decreasing and integrable $\sum_{k\geq 1}f(k)<\infty$. Set $l_k = c_0 f(k)$ with $c_0$ chosen so that $\sum_{k\geq 1}l_k=A$. Define $M(\tau)$ to be the smallest integer such that $M(\tau)\geq f^{-1}(\tau)=h(\tau)$. To complete the proof of the corollary we need to relate the quantities in the statement to the corresponding quantities in Proposition~\ref{prop: Aux1 integrals}.

  Since $f$ is monotone decreasing we can estimate
  \begin{align}\label{eq: l1, l2 bounds}
    \sum_{k\geq M(\tau)} f(k) &\geq \int_{M(\tau)}^\infty f(y)\, dy= \int_{h(\tau)}^\infty f(y)\, dy - \int_{h(\tau)}^{M(\tau)}f(y)\, dy
    \geq 
    \int_{h(\tau)}^\infty f(y)\, dy - \tau\, , \\[5pt]
    \sum_{k\geq M(\tau)} f^2(y) 
    & \leq
    \int_{M(\tau)}^\infty f^2(y)\, dy + f^{2}(M(\tau))
    \leq
    \int_{h(\tau)}^\infty f^2(y)\, dy + \tau^2\, .
  \end{align}

  By the change of variables $x=f(y)$, i.e.\ $y=h(x)$, and an integration by parts
  \begin{equation}\label{eq: int by parts l1}
    \int_{h(\tau)}^\infty f(y)\, dy = -\int_0^\tau x h'(x)\, dx = \int_0^\tau h(x)\, dx -\tau h(\tau)\, , 
  \end{equation}
  and similarly
  \begin{equation}\label{eq: int by parts l2}
    \int_{h(\tau)}^\infty f^2(y)\, dy = -\int_0^\tau x^2 h'(x)\, dx = 2\int_0^\tau x h(x)\, dx -\tau^2h(\tau)\, .
  \end{equation}

  Combining \eqref{eq: l1, l2 bounds}--\eqref{eq: int by parts l2} with the properties of $h$ in Proposition~\ref{prop: Aux1 integrals} and $G(\tau)\geq \tau$ we find
  \begin{equation*}
    \lim_{\tau\to 0^\limplus} \frac{\sum_{k\geq M(\tau)}l_k}{G(\tau)}
    \geq
    c_0 \lim_{\tau\to 0^\limplus} \frac{\int_0^\tau h(x)\, dx -\tau(h(\tau)+1)}{G(\tau)} = \infty\, , 
  \end{equation*}
  and, since $M(\tau)\leq h(\tau)+1$, 
  \begin{equation*}
    0\leq \lim_{\tau\to 0^\limplus} \frac{\tau M(\tau)+ \frac{1}{\tau}\sum_{k\geq M(\tau)}l_k^2}{\sum_{k\geq M(\tau)}l_k} 
    \leq
    \lim_{\tau\to 0^\limplus} \frac{c_0^{-1}\tau(h(\tau)+1) + \frac{c_0}{\tau}\bigl(2\int_0^\tau xh(x)\, dx +\tau^2\bigr)}{\int_0^\tau h(x)\, dx-\tau(h(\tau)+1)}=0\, .
  \end{equation*}
  This completes the proof of Corollary~\ref{cor: Aux2 sequences}.
\end{proof}

\section{Proof of Theorem~\ref{thm:MainRiesz}}
\label{sec:ThmHeat implies ThmRiesz}

In the final section of this note we provide a proof that Theorem~\ref{thm:MainHeat} implies Theorem~\ref{thm:MainRiesz}.

\begin{proof}[Proof of Theorem~\ref{thm:MainRiesz}]
Fix $R$ as in the theorem. Without loss of generality, we may assume that $R$ is bounded. 

Let
$$
g_{0}(t) = \int_0^\infty \mu^{(d-1)/2}R(\mu/t)  e^{-\mu/2}\, d\mu \, 
$$
and $\tilde g_0(t)=\max\{g_0(t), t^{(d+1)/2}\}$. By dominated convergence, $\lim_{t\to0^\limplus} \tilde g_0(t) = 0$. By Theorem~\ref{thm:MainHeat}, there exists an open, bounded, and connected Lipschitz regular set $\Omega_0\subset\R^d$ such that
$$
\limsup_{t\to 0^\limplus} \frac{(4\pi t)^{d/2} \Tr(e^{t\Delta_{\Omega_0}}) - |\Omega_0| + \frac{\sqrt{\pi t}}{2} \mathcal H^{d-1}(\partial\Omega_0)}{\sqrt t \tilde g_0(t)} = \infty \, .
$$

Assume the conclusion of Theorem~\ref{thm:MainRiesz} is wrong. Then there exist $C<\infty$, $\gamma \geq 0$, and $0\leq \lambda_0<\infty$ such that for all $\lambda\geq \lambda_0$, 
$$
\Tr(-\Delta_{\Omega_0}-\lambda)_\limminus^\gamma - L_{\gamma, d}|\Omega_0| \lambda^{\gamma+d/2} + \frac{L_{\gamma, d-1}}{4} \mathcal H^{d-1}(\partial\Omega_0) \lambda^{\gamma+(d-1)/2} \leq C \lambda^{\gamma+(d-1)/2}R(\lambda)\,.
$$

By~\eqref{eq:Laplace transform} we conclude that
\begin{align*}
& \Tr(e^{t\Delta_{\Omega_0}}) - (4\pi t)^{-d/2} |\Omega_0| + \frac{\sqrt{\pi t}}{2} (4\pi t)^{-d/2} \mathcal H^{d-1}(\partial\Omega_0) \\
& = \frac{t^{1+\gamma}}{\Gamma(1+\gamma)} \int_0^\infty \biggl[\Tr(-\Delta_{\Omega_0}-\lambda)_\limminus^\gamma - L_{\gamma, d}|\Omega_0| \lambda^{\gamma+d/2} + \frac{L_{\gamma, d-1}}{4} \mathcal H^{d-1}(\partial\Omega_0) \lambda^{\gamma+(d-1)/2} \biggr] e^{-t\lambda}\, d\lambda \\
& \leq \frac{t^{1+\gamma}}{\Gamma(1+\gamma)} \biggl[C\int_{\lambda_0}^\infty \lambda^{\gamma+(d-1)/2} R(\lambda) e^{-t\lambda}\, d\lambda +  c(t)\biggr]\,,
\end{align*}
where we write
\begin{equation*}
c(t)=\int_0^{\lambda_0}\! \biggl[\Tr(-\Delta_{\Omega_0}-\lambda)^\gamma_\limminus - L_{\gamma, d}|\Omega_0| \lambda^{\gamma+d/2} + \!\frac{L_{\gamma, d-1}}{4} \mathcal H^{d-1}(\partial\Omega_0) \lambda^{\gamma+(d-1)/2}\biggr] e^{-t\lambda}\, d\lambda\, .
\end{equation*}
Using the fact that $\mu^\gamma e^{-\mu/2}\leq \bigl(\frac{2\gamma}{e}\bigr)^\gamma$ for all $\mu\geq 0$,
\begin{align*}
  t^{1+\gamma} \int_{\lambda_0}^\infty \lambda^{\gamma+(d-1)/2} R(\lambda) e^{-t\lambda}\, d\lambda 
  &=
  t^{-(d-1)/2} \int_{\lambda_0 t}^\infty \mu^{\gamma+(d-1)/2} R(\mu/t) e^{-\mu}\, d\mu\\
  &\leq
  \Bigl(\frac{2\gamma}{e}\Bigr)^\gamma t^{-(d-1)/2} \int_{\lambda_0 t}^\infty \mu^{(d-1)/2} R(\mu/t) e^{-\mu/2}\, d\mu\\
  &\leq
  \Bigl(\frac{2\gamma}{e}\Bigr)^\gamma t^{-(d-1)/2} g_0(t)\\
  &\leq
  \Bigl(\frac{2\gamma}{e}\Bigr)^\gamma t^{-(d-1)/2}\tilde g_0(t)\, .
\end{align*}
Moreover, we bound $c(t)$ by discarding the negative volume term and use monotonicity to find
\begin{align*}
c(t) & \leq \biggl( \Tr(-\Delta_{\Omega_0}-\lambda_0)^\gamma_\limminus + \frac{L_{\gamma, d-1}}{4} \mathcal H^{d-1}(\partial\Omega_0) \lambda_0^{\gamma+(d-1)/2} \biggr) \int_0^{\lambda_0} e^{-t\lambda}\, d\lambda \\
& \leq \biggl( \Tr(-\Delta_{\Omega_0}-\lambda_0)^\gamma_\limminus + \frac{L_{\gamma, d-1}}{4} \mathcal H^{d-1}(\partial\Omega_0) \lambda_0^{\gamma+(d-1)/2} \biggr) \lambda_0 \, .
\end{align*}
By construction $\tilde g_0(t) \geq t^{(d+1)/2}$, and therefore
\begin{equation*}
\limsup_{t\to 0^\limplus}\, t^{\gamma+(d+1)/2} \tilde g_0(t)^{-1} c(t) <\infty \, .
\end{equation*}
Thus we have shown that 
\begin{equation*}
  \limsup_{t\to 0^\limplus} \frac{(4\pi t)^{d/2}\Tr(e^{t\Delta_{\Omega_0}})-|\Omega_0|+ \frac{\sqrt{\pi t}}{2}\Haus^{d-1}(\partial\Omega_0)}{\sqrt{t}\tilde g_0(t)}<\infty\,,
\end{equation*}
contradicting the choice of $\Omega_0$. This completes the proof of Theorem~\ref{thm:MainRiesz}.
\end{proof}

\bibliographystyle{amsalpha}

\end{document}